\documentclass[letterpaper, 11pt]{article}
\usepackage{amsmath, amsthm, amssymb}
\usepackage{graphicx}
\usepackage{xcolor}
\usepackage{algorithm}
\usepackage{algpseudocode}
\usepackage[colorlinks,linkcolor=blue,citecolor=blue,urlcolor=magenta,linktocpage,plainpages=false]{hyperref}
\usepackage{caption}
\usepackage{subcaption}
\usepackage{setspace}
\usepackage[left=1in, right=1in, top=1in, bottom=1in]{geometry}
\usepackage{authblk}
\usepackage{changepage}
\usepackage{enumitem}

\newtheorem{theorem}{Theorem}

\newtheorem{claim}{Claim}

\title{Lower bound on size of branch-and-bound trees for solving lot-sizing problem}
\author[1]{Santanu S. Dey\thanks{santanu.dey@isye.gatech.edu}}
\author[1]{Prachi Shah\thanks{prachi.shah@gatech.edu}}
\affil[1]{School of Industrial and Systems Engineering, Georgia Institute of Technology}
\date{}

\begin{document}

\maketitle
\begin{abstract} We show that there exists a family of instances of the lot-sizing problem, such that any branch-and-bound tree that solves them requires an exponential number of nodes, even in the case when the branchings are performed on general split disjunctions. 
\end{abstract}
\section{Introduction}\label{sec:into}
\subsection{Branch-and-bound procedure}
Land and Doig~\cite{land1960automatic} invented the branch-and-bound procedure to solve  mixed integer linear programs (MILP). Today, all state-of-the-art  MILP solvers are based on the branch-and-bound procedure.  An important decision is formalizing a branch-and-bound procedure is to decide the method to partition the feasible region of the linear program corresponding to a node. If the partition is based on variable disjunctions, that is, the feasible region of the linear program at a given node is partitioned by adding the inequality of the form $x_i \leq \eta$ to one child node and the inequality  $x_i \geq \eta + 1$  to the other child node where $\eta$ is an integer, then we call the branch-and-bound tree as a \emph{simple} branch-and-bound tree. On the other hand if we allow the use of more general split disjunctions of the form:
$$\left(\pi^{\top}x \leq \eta \right) \vee \left(\pi^{\top}x \geq \eta + 1 \right),$$
where $\pi$ is an integer vector and $\eta$ is an integer, to create two child nodes, we call the resulting branch-and-bound tree as a \emph{general} branch-and-bound tree.  Clearly general branch-and-bound tree are expected to be smaller than simple branch-and-bound tree. However, in practice, MILP solvers use simple branch-and-bound trees (one reason may be to maintain the sparsity of linear programs solved at child nodes; see discussion in~\cite{dey2021branch,dey2015approximating}). 

One way to measure the efficiency of the branch-and-bound algorithm for a given class of instances, is to estimate the size of the branch-and-bound tree to solve the instances, since the size of the branch-and-bound tree corresponds to the number of linear programs solved. 

\paragraph{Upper bounds on size of branch-and-bound trees.} Pataki~\cite{pataki2010basis} showed that for certain classes of random integer programs the general branch-and bound-tree has linear number of nodes with high probability, while recently~\cite{dey2021branch,borst2020integrality} showed that for two different classes of random integer programs even the simple branch-and bound-tree has polynomial size (number of nodes), with good probability. The paper~\cite{dey2021theoretical} presents upper bounds on size of simple branch-and-bound tree for the vertex cover problem.

\paragraph{Lower bounds on size of branch-and-bound trees.} The papers~\cite{jeroslow1974trivial,chvatal1980hard} present examples of integer programs where every {simple} branch-and-bound algorithm for solving them has exponential size, although these instances can be solved using polynomial-size general branch-and-bound trees; see~\cite{yang2021multivariable,basu2020complexity2}. Cook et al.~\cite{cook1990complexity} present a travelling salesman problem (TSP) instance that requires exponential-size branch-and-bound tree to solve when using simple branching. 

Note that a lower bound on the size of a general branch-and-bound tree is also a lower bound on the size of a simple branch-and-bound tree. The paper~\cite{dadush2020complexity}  was the first to prove an exponential lower bound on the size of general  branch-and-bound tree to prove the infeasibility of the cross-polytope. The paper~\cite{basu2020complexity2} shows that the sparsity of the disjunctions used for branching can have a large impact on the size of the branch-and-bound tree. The paper~\cite{dey2021lower} presents exponential lower bounds on the size of general branch-and-bound trees for solving a particular packing integer program, a particular covering integer program, and a particular TSP instance. 

This paper contributes to this literature, by showing an (worst case) exponential lower bound on the size of general branch-and-bound tree for solving lot-sizing problems.   

\subsection{Lot-sizing problem}
In this paper, we consider the classical lot-sizing problem of determining production volumes to meet demands in $n$ periods exactly, while minimizing production cost and fixed cost of production. A lot-sizing problem with a time horizon of $n$ periods can be formulated as a mixed-integer linear program (MILP) as follows,
\begin{subequations}
\begin{eqnarray} 
& \textup{min} & \sum_{i=1}^{n} p_i \, x_i + \sum_{i=1}^{n} f_i \, y_i  \label{lotsize:obj} \\
&& \sum_{k = 1}^i x_{k} \geq d_{1,i}   \ \textup{for all} \ i \in \{1, \dots, n- 1\},  \label{lotsize:flow} \\
&& \sum_{k = 1}^n x_{k} = d_{1,n},  \label{lotsize:flow1} \\
&& x_i \le d_{i,n} \, y_i \ \textup{for all} \ i \in \{1, \hdots, n\}, \label{lotsize: x-y} \\
&& x \in  \mathbb{R}^n_+, \  y \in  [0, 1 ]^n, \label{lotsize:bounds} \\
&& y \in \mathbb{Z}^n,
\end{eqnarray}
\end{subequations}\label{eq:milp_lotsizing}
where variable $x_i$ is the quantity produced in period $i$ and $y_i$ is a binary variable with value $1$ if production occurs in period $i$ and $0$ otherwise. Unit cost of production, fixed cost of production, and demand in period $i$ are denoted by $p_i$, $f_i$ and $d_i$ respectively. We represent the cumulative demand from period $i$ to period $j$ by, 
\begin{equation*}
    d_{i,j} = \left(\sum_{k=i}^j d_k \right).
\end{equation*}
The lot-sizing problem is a very well-studied problem~\cite{pochet2006production} with many important applications. The classical dynamic programming algorithm to solve lot-sizing problem uses the so-called Wagner-Whitin property and runs in $\mathcal{O}(n^2)$~\cite{wagner1958dynamic}. This running-time was later improved to  $\mathcal{O}(n \textup{log}(n))$~\cite{aggarwal1993improved,federgruen1991simple,wagelmans1992economic}. The full polyhedral description of the convex hull of feasible solutions is presented in~\cite{barany1984uncapacitated}.

In this paper, we show that even though lot-sizing is such a ``simple" problem,  that is there is a polynomial-time dynamic programming algorithm to solve it, general branch-and-bound tree in the worst case may take exponential-time to solve the problem.  Formally we prove the following:
\begin{theorem}\label{thm:1}
Consider the lot-sizing instance on $n$ time periods with 
\begin{equation}
\begin{gathered} \label{exp_instance}
    f_j = 1, \quad p_j = n - j + 1, \quad d_j = 1 \quad \textup{for all } j \in  \{1, \dots, n\}.
\end{gathered}
\end{equation}
Then any general branch-and-bound tree that solves this instance has at least $2^{ (n/2) -1}$ leaf nodes.
\end{theorem}
We provide a proof of Theorem~\ref{thm:1} in the next section. 

\section{Proof of Theorem~\ref{thm:1}}

We begin by finding the optimal objective function value for the class of instances (\ref{exp_instance}).
\begin{claim} \label{claim:mip opt val}
For the lot-sizing instance (\ref{exp_instance}) with $n$ time periods, the optimal objective function value is $\dfrac{n(n+1)}{2} + n$.
\end{claim}
\begin{proof}
Observe that $y_1 = 1$ for any feasible solution since $d_1 > 0$. Now, consider the demand at period $j \geq 2$ and $y_j$. There are two possible cases: 
\begin{enumerate}[label=(\alph*)]
    \item $y_j = 1$. Then we may assume that the demand at $j$ is met by production in period $j$~\cite{pochet2006production}. The total cost incurred for satisfying demand at $j$ is, 
    \begin{equation*}
        f_j + p_j = n - j + 2.
    \end{equation*}
    \item $y_j = 0$. Then we may assume that the entire demand at $j$ is met by production in some period $i < j$~\cite{pochet2006production}. In this case, the unit cost of production of at period $i$ is 
    \begin{equation*}
        p_i = n - i + 1 \ge n - j + 2.
    \end{equation*}
\end{enumerate}
In other words, the increase in unit cost incurred by producing in an earlier time period, is at least as much as the fixed cost 
for period $j$. Therefore, we conclude that setting  $y_j = 1 \; \textup{for all } j \in \{1, \dots, n\}$ leads to an optimal solution. 
This optimal solution is:
\begin{align*}
    \hat{x}_j &= d_j = 1,  \ \, \textup{for all } j \in \{1, \dots, n\},\\
    \hat{y}_j &= 1, \ \, \textup{for all } j \in \{1, \dots, n\}.
\end{align*}
Therefore, the optimal cost (OPT) of the MILP is  
\begin{equation} \label{eq:opt_milp}
    \begin{aligned}
        OPT =& \sum_{i=1}^{n} p_i \, \hat{x}_i + \sum_{i=1}^{n} f_i \, \hat{y}_i  \\
    = & \ \sum_{i=1}^{n} (n - j + 1) + \sum_{i=1}^{n} 1 =  \ \dfrac{n(n+1)}{2} + n.
    \end{aligned}
\end{equation}
\end{proof}

Given any node $\mathcal{N}$ in a branch-and-bound tree, we have two types of constraints defining the feasible region of the linear program corresponding to $\mathcal{N}$: (i) the constraints describing the original lot-sizing formulation (\ref{lotsize:flow})-(\ref{lotsize:bounds}), and (ii) the additional inequalities that were added to this node and its ancestors (starting at the child of the root node). We call the second set of constraints as branching constraints at node $\mathcal{N}$. 

We will now show the size of any general branch-and-bound tree for lot-sizing instances (\ref{exp_instance}) is exponential in the number of time periods, that is present a proof of Theorem~\ref{thm:1}. 

For simplicity of exposition, we first consider the case where $n$ is odd. Let $\mathcal{T}$ be any general branch-and-bound tree that solves (\ref{exp_instance}).  Let $\mathcal{S} := \left\{ y\in \{0, 1\}^ n \, | \, y_j = 1 \textup{ if j is odd} \right\} $ and note that $|\mathcal{S}| = 2^{(n-1)/2} \ge 2^{ (n/2) -1}$. Since the elements of $\mathcal{S}$ are integer vectors, they must satisfy the branching constraints of some leaf node. 

 Now, consider any $u, v \in \mathcal{S}$ such that $u \neq v$ and let $\mathcal{N}$ be a node of $\mathcal{T}$ such that both $u$ and $v$ are feasible for the branching constraints at $\mathcal{N}$.  We will show that $\mathcal{N}$ is not a leaf node of $\mathcal{T}$ by constructing a solution $(\hat{x}, \hat{y})$ that satisfies the constraints describing $\mathcal{N}$ and whose objective value is strictly better than the MILP optimal objective function value. Equivalently, any leaf node of $\mathcal{T}$ contains at most one element of $\mathcal{S}$. Thus $\mathcal{T}$ must have at least $|S|$ leaf nodes, completing the proof.

It remains to show that $\mathcal{N}$ is not a leaf node when $u,v$ satisfy the branching constraint of $\mathcal{N}$, where $u, v \in \mathcal{S}$ and $u \neq v$. Let $\hat{y} := \frac{(u+v)}{2}$. By convexity, $\hat{y}$ also satisfies the branching constraints at $\mathcal{N}$. 

Construct $\hat{x}$ as follows, 
\begin{align*}
    \hat{x}_j &= 
    \begin{cases} 
    1 & \textup{ if } j = n \\
    1 & \textup{ if } j<n \textup{ is odd and } \hat{y}_{j+1} \in \{1, 0.5\} \\
    2 & \textup{ if } j<n \textup{ is odd and } \hat{y}_{j+1} = 0 \\
    1 & \textup{ if } j \textup{ is even and } \hat{y}_{j} \in \{1, 0.5\} \\
    0 & \textup{ if } j \textup{ is even and } \hat{y}_{j} = 0. \\
    \end{cases} 
\end{align*}
It is straightforward to verify that $(\hat{x}, \hat{y})$ satisfies (\ref{lotsize:flow})-(\ref{lotsize:bounds}). Since $\hat{y}$ satisfies the branching constraints at $\mathcal{N}$, this shows that $(\hat{x}, \hat{y})$ is a feasible solution for linear program corresponding to $\mathcal{N}$.

We will now compute the objective function value of $(\hat{x}, \hat{y})$ by considering consecutive pairs of time periods, $j$ and $j+1$ for odd values of $j < n$. There are three possible cases: 
\begin{enumerate}[label=(\roman*)]
    \item $\hat{y}_{j+1} = 0 \; \implies \hat{x}_j = 2, \, \hat{x}_{j+1} = 0$: 
    \begin{equation*}
         OBJ_j + OBJ_{j +1} = ( 1 + (n-j+1)\cdot 2 ) + (0) = (n-j+1) + (n-j) + 2.
    \end{equation*}
    
    \item $\hat{y}_{j+1} = 1 \; \implies \hat{x}_j = 1, \, \hat{x}_{j+1} = 1$: 
    \begin{equation*}
         OBJ_j + OBJ_{j+1} = ( 1 + (n-j+1)\cdot 1 ) + ( 1 + (n-j)\cdot 1 ) = (n-j+1) + (n-j) + 2.
    \end{equation*}
    
    \item $\hat{y}_{j+1} = 0.5 \; \implies \hat{x}_j = 1, \, \hat{x}_{j+1} = 1$: 
    \begin{equation*}
         OBJ_j + OBJ_{j+1} = ( 1 + (n-j+1)\cdot 1 ) + ( 0.5 + (n-j)\cdot 1 ) = (n-j+1) + (n-j) + 2 - 0.5.
    \end{equation*}
\end{enumerate}
Lastly, note that the contribution to the objective function from the last period is $2$, i.e. $OBJ_n = 2$. Let $N_{0.5}$ be the number of coordinates of $\hat{y}$ that are 0.5 and $OPT$ be the optimal MILP objective value from Claim~\ref{claim:mip opt val}. Then, the objective function value for  $(\hat{x}, \hat{y})$ is, 
\begin{align*}
    \sum_{j=1}^n OBJ_j 
    &= \sum_{j=1}^{n-1} OBJ_j + OBJ_n \\
    &= \sum_{\substack{j=1, \\ j \textup{ is odd}}}^{n-2} \left( (n-j+1) + (n-j) + 2 \right) - 0.5\,N_{0.5} + 2\\
    &= \sum_{j=1}^{n} \left((n-j+1) + 1\right) - 0.5\,N_{0.5} \\
    &=  \dfrac{n(n+1)}{2} + n - 0.5\,N_{0.5} \\
    & = OPT - 0.5\,N_{0.5} \\
    & < OPT,
\end{align*}
where the last inequality follows since $\hat{y}$ must have at least one component equal to $0.5$ since $u \neq v$. Thus, $\mathcal{N}$ cannot be a leaf node. This completes the proof. 

In the case where $n$ is even, the proof follows similarly, by defining 
\begin{equation*}
\mathcal{S} := \left\{ y\in \{0, 1\}^ n \, | \, y_1 = 1, \, y_j = 1 \textup{ if j is even} \right\}    
\end{equation*}
and noting that $|\mathcal{S}| = 2^{ (n/2) -1}$.

\bibliographystyle{plain}
\bibliography{ref}

\end{document}